\documentclass[12pt]{article}
\usepackage{amssymb,bm,epsf,latexsym,a4}
\usepackage[english]{babel}
\def\bbbr{{\mathbb R}}

\def\bbbq{{\mathbb Q}}

\def\bbbz{{\mathbb Z}}
\def\boldalpha{\bm{\alpha}}

\def\v#1{{\bf#1}}

\newtheorem{theorem}[subsection]{Theorem}
\newtheorem{definition}[subsection]{Definition}
\newtheorem{lemma}[subsection]{Lemma}

\newtheorem{corollary}[subsection]{Corollary}

\newtheorem{proposition}[subsection]{Proposition}

\newenvironment{proof}{\noindent {\bf Proof}}{\hfill\break\medskip\hfill$\Box$\medskip}

\title{Geodesic continued fractions and LLL}
\author{{Frits Beukers}}
\date{October 6, 2013}

\begin{document}
\maketitle
\abstract{
We discuss a proposal for a continued fraction-like
algorithm to determine simultaneous rational approximations to
$d$ real numbers $\alpha_1,\ldots,\alpha_d$. It combines an
algorithm of Lagarias with ideas from LLL-reduction. We
dynamically LLL-reduce a quadratic form with parameter $t$ as
$t\downarrow0$.
Suggestions in this direction have been made many times over in the
literature, e.g. \cite[p 104]{hensley} or \cite{bosmasmeets}.
The new idea in this paper is that checking the LLL-conditions
consists of solving linear equations in $t$.}

\parindent=0pt

\section{Introduction}
Let $\boldalpha=(\alpha_1,\ldots,\alpha_d)\in\bbbr^d$
and suppose that not all $\alpha_i$ are rational. By application
of the pigeon-hole principle one can show that there exist
infinitely many $(\v p,q)=(p_1,p_2,\ldots,p_d,q)\in\bbbz^{d+1}$
with gcd one and $q>0$ such that
$$|\v p-q\boldalpha|\le \sqrt{d}q^{-1/d}\eqno{\rm(D)}.$$
Here $|\,.\,|$ denotes the Euclidean norm on $\bbbr^d$. An alternative
description is $||q\boldalpha||\le\sqrt{d}q^{-1/d}$ where $||\v x||$
denotes the distance of $\v x$ to the nearset lattice point. The
numbers $p_1,\ldots,p_d$ are simply the coordinates of such a nearest
lattice point.

One would like
to have an algorithm which computes such approximations. The Jacobi-Perron
algorithm and related versions of it (modified Perron, Brun, Selmer)
seem to yield upper bounds like $c\cdot q^{-\delta}$ but with $0<\delta<1/d$
for general $\boldalpha\in\bbbr^d$. So they are not expected to be very
good.

In a letter to Jacobi, Hermite explained another idea to construct good
simultaneous approximations, \cite[p106]{hermiteletter}. See also
\cite[p xii,xiii]{picard}. Choose $t>0$ and consider
the quadratic form
$$Q_t=(x_1-\alpha_1y)^2+\cdots+(x_d-\alpha_dy)^2+ty^2$$
in $x_1,\ldots,x_d,y$. Choose integers $p_1,\ldots,p_d,q$ as arguments
which minimize this form. Then Hermite was able to show that
$$|\v p-\boldalpha q|\le \gamma_dq^{-1/d}$$
where $\gamma_d$ is a number depending only on $d$.
For a proof see Proposition \ref{lagariasobservation} in this paper.
All that is required now, is a reduction algorithm that enables one to find
the minimizing set of integers.

In 1994 Jeff Lagarias, in \cite{lagarias}, took up this idea again
and proposed an algorithm which consists in decreasing $t$ to $0$ and along
the way perform coordinate changes so that the form remains Minkowski
reduced (see Section \ref{reduction} for a definition). The result is an
algorithm of the type sketched on page \pageref{geodesicalgorithm}.
A similar elaboration
in the case $d=1$ can already be found in a paper by Humbert, \cite{humbert}.
The result is an algorithm that produces good simultaneous approximations
like (D). In \cite{lagarias} Lagarias gives an analysis of this
algorithm. For example, it
finds best approximations in the Euclidean norm sense. That is, it
finds $q\in\bbbz_{>0}$ such that $||q\boldalpha||\le||q'\boldalpha||$ for
all integers $q'$ with $0<q'<q$. However, it is not guaranteed that all
of them are found.

A nice feature of Lagarias' algorithm is that the Minkowski reducedness conditions
are linear in $t$, which makes the check and update process easy. The disadvantage is
that the number of conditions grows prohibitively large as $d$ increases.
Already for $d=7$ about 90,000 conditions are needed.

This problem might be circumvented by the use of LLL-reduction instead of Minkowki
reduction (again see Section \ref{reduction} for a definition). Since the LLL-algorithm
gives suboptimal results, one cannot expect to find guaranteed best approximations.
If one is willing to accept this, another potential
problem is that the LLL-reduction conditions are non-linear in the coefficients
of $Q_t$, thus making their verification difficult.
However, the main contribution of this paper is the observation is that
the conditions are still linear in $t$. The algorithm we propose in Section
\ref{algorithm} is not very surprising, but its feasibility is based on
the linearity in $t$ of the LLL-conditions. We have not carried out any
experiments yet to see if the algorithm is practical in any sense.

We remark that the same idea and results would also work in finding small values of
$|q+p_1\alpha_1+\cdots+p_d\alpha_d|$. One would have to use the
family of forms $p_1^2+\cdots+p_d^2+t(q+p_1\alpha_1+\cdots+p_d\alpha_d)^2$
with $t\uparrow\infty$.

{\bf Acknowledgements} Many thanks to Robbert Fokkink and Cor Kraaikamp for
their invation to the workshop 'Probability and Numbers' in Delft. Thanks
also to Catherine Goldstein who provided me with a number of very interesting
references to the work of Charles Hermite.

\section{Quadratic forms}
A quadratic form in the variables $x_1,x_2,\ldots,x_n$ is a homogeneous
quadratic polynomial with coefficients in $\bbbr$. We write such a form
in the shape
$$\sum_{i,j=1}^n q_{ij}x_ix_j\qquad \forall i,j: q_{ij}=q_{ji}.$$
Very often we abbreviate this to $Q(\v x)$. Without causing too much
confusion we also use the notation $Q$ for the $n\times n$-matrix with
entries $q_{ij}$. We call this the matrix associated to the quadratic form
and the absolute value of the determinant of $Q$ is called the {\it determinant}
of the form.

The form $Q(\v x)$ is called {\it positive definite} if $Q(\v x)\ge0$ for
all $\v x\in\bbbr^n$ and $Q(\v x)=0\iff \v x=\v 0$. From now on, when we
speak of form, we mean a positive definite quadratic form.

For us, an important question will to determine the minimal non-zero value
of the set $\{Q(\v x)|\v x\in\bbbz^n\}$, which we denote by $\mu(Q)$. We have the
following theorem (see \cite[Ch 12]{cassels}).
\begin{theorem}[Hermite]\label{hermite} For every $n\ge2$ there exists $\gamma_n$ such
that $\mu(Q)\le \gamma_nD(Q)^{1/n}$ for all positive definite forms $Q$ in
$n$ variables.
\end{theorem}
The smallest possible values of $\gamma_n$ are known as {\it Hermite's constants}.
We again denote them by $\gamma_n$. The first few values are
$\gamma_2=2/\sqrt{3},\gamma_3=2^{1/3},\gamma_4=\sqrt{2},\ldots$. In general
we have $\gamma_n\le 2n/3$.

It is also interesting to consider the so-called successive minima of a form.
The $i$-th successive minimum $\mu_i(Q)$ is defined as the smallest real number
$\mu$ such that the ball defined by $Q(\v x)\le\mu$ contains a set of $i$ independent
vectors from $\bbbz^n$. In particular, $\mu_1(Q)=\mu(Q)$.

Another feature of forms in $n$ variables is that the space of forms can be
identified with the Riemannian symmetric space $O(n,\bbbr)\backslash GL(n,\bbbr)$.
The correspondence is given by the map $g\in GL(n,\bbbr)\mapsto g^Tg$,
where $g^T$ denotes the transpose of $g$. Notice that $g_1^Tg_1=g_2^Tg_2$ if
and only if there exists $u\in O(n,\bbbr)$ such that $g_2=ug_1$. A metric
on $O(n,\bbbr)\backslash GL(n,\bbbr)$ is given by
$$ds^2={\rm trace}((dY\cdot Y^{-1})\cdot(dY\cdot Y^{-1})^T)$$
where $Y\in GL(n,\bbbr)$. The geodesics with respect to this metric are
the one-dimensional families of quadratic forms
$$e^{\lambda_1s}l_1(\v x)^2+e^{\lambda_2s}l_2(\v x)^2+\cdots+e^{\lambda_ns}l_n(\v x)^2$$
parametrized by $s$, where $l_1,\ldots,l_n$ are independent forms.
Thus we see that the family $Q_t$ is a
geodesic in the space of forms in $d+1$ variables, which accounts
for the name 'geodesic algorithm'.

\section{Reduction of forms}\label{reduction}
Two forms $Q,\tilde{Q}$ in $n$ variables are said to be equivalent if there
exists $g\in GL(n,\bbbz)$ such that $\tilde{Q}(\v x)=Q(g\v x)$. It is common
practice to choose suitably nice elements in each equivalence class, which we
call reduced forms. There exist several notions of reduced forms, but for us
the following two will be relevant: Minkowski reduced forms and LLL-reduced forms.

A form $Q(\v x)$ is called {\it Minkowski reduced} if for $i=1,2,3,\ldots,n$ we
have
$$Q(\v e_i)\le Q(\v m)\ {\rm for\ all}\ \v m\in\bbbz\ {\rm with}\ \gcd(m_i,\ldots,m_n)=1
\eqno{(M)}$$
Here $\v e_1,\v e_2,\ldots,\v e_n$ is the standard basis of $\bbbr^n$.
Another way of stating these inequalities is to say that $Q(\v e_i)$ is the minimum
of all $Q(\v m)$ with $\v m\in\bbbz^n$ such that $\v e_1,\ldots,\v e_{i-1},\v m$ can be
extended to a basis of $\bbbz^n$.
In particular $Q(\v e_1)=q_{11}$
is the smallest non-zero value of $Q$ restricted to $\bbbz^n$. The value $Q(\v e_2)=q_{22}$
is smallest value of all $\v x\in\bbbz^n$ independent of $\v e_1$. However, it is not always
true that for a Minkowksi reduced form $Q(\v e_i)$ is the smallest value of $Q$ at all
$\v x\in\bbbz^n$ independent of $\v e_1,\ldots,\v e_{i-1}$. The smallest value of $i$ for which
this fails is $i=5$.

We denote the set of Minkowski reduced forms by ${\cal M}_n$. We have the following properties
(see \cite[Ch 12]{cassels}).
\begin{proposition}Let notations be as above. Then,
\begin{enumerate}
\item Every equivalence class of forms contains an element in ${\cal M}_n$.
\item Every equivalence class of forms contains finitely many forms in ${\cal M}_n$.
\item Two forms in the interior of ${\cal M}_n$ can only be equivalent via
trivial substitutions of the form $x_i\to \epsilon_i x_i$ for all $i$, where
$\epsilon_i\in\{\pm1\}$.
\item For every $Q\in{\cal M}_n$ we have
$$\mu_i(Q)\le Q(\v e_i)\le 2^i\mu_i(Q),\quad i=1,\ldots,n.$$
\item The space ${\cal M}_n$ can be defined by a finite number of inequalities
of the form (M).
\end{enumerate}
\end{proposition}

We illustrate the last fact for the cases $n=2,3$ (see \cite[Ch 12, Lemma 1.2]{cassels}.
When $n=2$ the form
reads $ax^2+2bxy+cy^2$ with associated matrix
$$\pmatrix{a & b\cr b & c\cr}.$$
The form is positive definite if and only if $a>0,b^2-ac<0$. Its determinant reads $ac-b^2$.
It is not hard to show that the form is Minkowski reduced if
$$|b|\le a\le c.$$
A general ternary form (the case $n=3$) reads
$$ax^2+2bxy+2cxz+dy^2+2eyz+fz^2.$$
One can show that it is Minkowski-reduced if and only if
\begin{eqnarray*}
&&a\le d\le f,\quad |b|\le a,|c|\le a,|e|\le d\\
&&a+d\ge (\pm b\pm c\pm e),\quad {\rm zero\ or\ two\ minus\ signs}.
\end{eqnarray*}

Having a Minkowski-reduced form equivalent to our given form $Q$ yields a precious
amount of information on $Q$. For example, the coefficient of $x_1^2$ of the reduced
form is precisely $\mu(Q)$. Therefore it is of interest to find procedures that
produce a reduced form equivalent to $Q$.
In the case $n=2$ there is already the well-known reduction procedure by Gauss.
For other small values of $n$ one can also device reduction
procedures which are based on the inequalities that characterize Minkowksi reducedness.
Unfortunately it turns out that already when $n=7$, the number of inequalities has risen to about
90,000. So it is clear that for $n>6$ Minkowski reduction procedures tend to become
unwieldy. Nevertheless, there are a number of papers in which one proposes Minkowski
reduction algorithms for higher dimensions, see for example \cite{afflerbachgrothe},
\cite{helfrich}, \cite{zhangqiaowei}

There is another concept of reducedness which avoids the exponential growth
of reduction conditions, but at the cost of non-optimal output.
It is called LLL-reduction, named after its inventors Laszlo Lovasz, Arjen Lentra and
Hendrik Lenstra, who proposed it in 1982, \cite{LLL}. The corresponding reduction algorithm that
belongs to it has been extremely successful in many applications. It is simple, fast,
even in large dimension, and yields good results.

To define LLL-reducedness we write $Q$ in the form
\begin{eqnarray*}
Q(\v x)&=&b_1(x_1+\mu_{12}x_2+\cdots+\mu_{1n}x_n)^2\\
&&+b_2(x_2+\mu_{23}x_3+\cdots+\mu_{2n}x_n)^2\\
&&\vdots\\
&&+b_{n-1}(x_{n-1}+\mu_{n-1,n}x_n)^2+b_nx_n^2.
\end{eqnarray*}

We say that we have written $Q$ in {\it recursive form}.
\begin{definition}\label{LLLconditions}
Fix a number $\omega\in[3/4,1]$ (slack factor). We call the form $Q$
{\it LLL-reduced} if
\begin{enumerate}
\item $|\mu_{ij}|\le 1/2$ for all $i<j$.
\item $\omega b_{i}\le b_{i+1}+\mu_{i,i+1}^2b_i$ for all $i<n$.\\
(Lovasz condition)
\end{enumerate}
\end{definition}

Using the recursive form Hermite already defines a notion of reduction,
\cite[p 122 ff]{hermiteletter}. 
A form $Q$ is called (Hermite) reduced if either one of the following
holds,
\begin{itemize}
\item $n=1$.
\item When $n>1$, we have $b_1=\mu(Q)$, $|\mu_{1j}|\le1/2$ for $j=2,\ldots,n$
and the form $Q-b_1(x_1+\mu_{12}x_2+\cdots+\mu_{1n}x_n)^2$ in $x_2,\ldots,x_n$ is reduced.
\end{itemize}
Nowadays
it is often called Hermite-Korkine-Zolotarev (HKZ) reducedness. One can easily show
that HKZ-reducedness implies that $b_{i+1}+\mu_{i,i+1}^2b_i\ge b_i$ for all $i<n$. 
So Lovasz condition can be seen as a relaxed version of this inequality
(when $\omega<1$).

In the literature LLL-reducedness is usually formulated in terms of lattice bases.
In this paper we consider a version which is in terms of quadratic forms.
Naively speaking one might think that reducedness of the quadratic form
entails $|\mu_{ij}|\le 1/2$ for all $i<j$, which we have seen earlier,
and the condition $b_1\le b_2\le\cdots \le b_n$. This is not going to work
however. The innovation of LLL is to replace the naive condition $b_i\le b_{i+1}$
by the Lovasz condition
given above. Let us denote $\tilde{b}_{i}=b_{i+1}+\mu_{i,i+1}^2b_{i}$. Then
one can easily verify that if we swap the variables $x_i,x_{i+1}$ in $Q(\v x)$
and rewrite the new form in recursive form again, the coefficient of
$(x_i+\cdots)^2$ is precisely $\tilde{b}_{i}$. The coefficient of $(x_{i+1}+\cdots)^2$
in the new form is $b_{i+1}b_i/\tilde{b}_{i}$.

An LLL-reduced form has many interesting properties.

\begin{theorem}[LLL]\label{LLLproperties}
Let $Q$ be a positive definite form in $n$ variables
and suppose $Q$ is LLL-reduced with $\omega=3/4$. Then
\begin{enumerate}
\item $D(Q)\le \prod_{i=1}^nQ(\v e_i)\le 2^{n(n-1)/2}D(Q).$
\item $Q(\v e_1)\le 2^{(n-1)/2}D(Q)^{1/n}.$
\item $Q(\v e_1)\le 2^{n-1}\mu(Q)$.
\item For $k=1,2,\ldots,n$ and all $j\le k$ we have
$$Q(\v e_j)\le 2^{n-1}\mu_k(Q).$$
\end{enumerate}
\end{theorem}
Since the proofs in the literature are usually given for lattice bases we
reproduce a proof valid for forms here.
\medskip

\begin{proof}:
First let us note that
$$Q(\v e_i)=b_i+b_{i-1}\mu_{i-1,i}^2+\cdots +b_1\mu_{1i}^2.$$
Secondly,
$$D(Q)=\prod_{i=1}^n b_i.$$
Thirdly, it follows from Lovasz condition that $b_{i}\ge (\omega-\mu_{i-1,i}^2)b_{i-1}\ge {1\over2}b_{i-1}$.

By repeated application of the third inequality we find that $b_i\ge 2^{j-i}b_j$ whenever $j\le i$,
hence $b_j\le 2^{i-j}b_i$.
Together with our first observation this implies
\begin{eqnarray*}
Q(\v e_i)&\le&b_i+{1\over 4}(b_{i-1}+\cdots+b_1)\\
&\le& b_i+{1\over4}(2+\cdots+2^{i-1})b_i\le 2^{i-1}b_i.
\end{eqnarray*}
The first assertion of our theorem follows from
$$D(Q)=\prod_{i=1}^nb_i\le \prod_{i=1}^nQ(\v e_i)\le \prod_{i=1}^n2^{i-1}b_i=2^{n(n-1)/2}D(Q).$$
To prove the second assertion we observe that for all $j\le i$,
$$Q(\v e_j)\le 2^{j-1}b_j\le 2^{j-1}\cdot 2^{i-j}b_i=2^{i-1}b_i.$$
Applied to the case $j=1$ this gives
$$Q(\v e_1)^n\le \prod_{i=1}^n 2^{i-1}b_i=2^{n(n-1)/2}D(Q).$$
Taking the $n$-th roots gives our assertion.

The third assertion is a special case of the fourth, so we restrict to the
fourth. Take a set of independent $\v x_1,\ldots,\v x_k\in\bbbz^n$
such that
$$\max(Q(\v x_1),\ldots,Q(\v x_k))=\mu_k(Q).$$
Choose $l$ minimal so that $\v x_1,\ldots,\v x_k$ lies in the
span of $\v e_1,\ldots,\v e_l$. Since the $\v x_i$ are independent we have
$l\ge k$. Suppose $\v x_i$ has $l$-th coordinate $\ne0$. Denote this
coordinate by $\xi$. Then $\xi$ is a no-zero
integer and we trivially get $Q(\v x_i)\ge b_l\xi^2\ge b_l$. So whenever $j\le l$,
$$Q(\v e_j)\le 2^{l-1}b_l\le Q(\v x_i)\le 2^{n-1}\mu_k(Q).$$
In particular, since $j\le l$, this assertion holds for all $j\le k$.
\end{proof}
\medskip

For later purposes we also introduce partial LLL-reduction.
We call the form $Q$ {\it partially LLL-reduced} if
\begin{enumerate}
\item $|\mu_{i,i+1}|\le 1/2$ for all $i<n$.
\item $\omega b_{i}\le b_{i+1}+\mu_{i,i+1}^2b_i$ for all $i<n$.\\
(Lovasz condition)
\end{enumerate}

We give a short description of the LLL-reduction
algorithm for quadratic forms. There are two operations, a {\it shift}
and a {\it swap}. A shift is a substitution of the form $x_r\to x_r+ax_s$
where $s>r$ and $a\in\bbbz$ is chosen such that the resulting $\mu_{rs}$
satisfies $|\mu_{rs}|\le1/2$. A swap simply interchanges two
neighbouring variables $x_r,x_{r+1}$.

\begin{proposition}\label{substitutionsI}
Let $Q$ be a form. We perform a shift or a swap and
denote the resulting form by $\tilde{Q}$. Denote the parameters of the
recursive form of $\tilde{Q}$ by $\tilde{b}_i$ and $\tilde{\mu}_{ij}$.
Suppose we perform a shift $x_r\to x_r+ax_s$. Then
\begin{enumerate}
\item $\tilde{b}_i=b_i$ for all $i$.
\item $\tilde{\mu}_{is}=\mu_{is}+a\mu_{ir}$ for $i=1,\ldots,r-1$
\item $\tilde{\mu}_{rs}=\mu_{rs}+a$
\item $\tilde{\mu}_{ij}=\mu_{ij}$ for all other $i,j$.
\end{enumerate}
Suppose we perform a swap $x_r\leftrightarrow x_{r+1}$. Then
\begin{enumerate}
\item $\tilde{b}_{r}=b_{r+1}+\mu_{r,r+1}^2b_{r}$.
\item $\tilde{b}_{r+1}=b_{r}b_{r+1}/\tilde{b}_{r}$.
\item $\tilde{b}_{i}=b_i$ for all $i\ne r,r+1$.
\item $\tilde{\mu}_{ir}=\mu_{i,r+1}$ for $i<r$.
\item $\tilde{\mu}_{i,r+1}=\mu_{i,r}$ for $i<r$.
\item $\tilde{\mu}_{r,r+1}=b_{r}\mu_{r,r+1}/\tilde{b}_{r}$.
\item $\tilde{\mu}_{rj}=(b_r\mu_{r,r+1}\mu_{rj}+b_{r+1}\mu_{r+1,j})/\tilde{b}_{r}$ for $j>r+1$.
\item $\tilde{\mu}_{r+1,j}=\mu_{rj}-\mu_{r,r+1}\mu_{r+1,j}$ for $j>r+1$.
\item $\tilde{\mu}_{ij}=\mu_{ij}$ for all other $i,j$.
\end{enumerate}
\end{proposition}

\begin{proof}: Straightforward computation.
\end{proof}

By a {\it global shift} we mean a sequence of shifts $x_i\to x_i+ax_j$
(with different $a$'s) after which $|\mu_{ij}|\le1/2$ for all $i<j$.
Here is a possible implementation of LLL-reduction.
\begin{enumerate}
\item For $i=1$ to $n-1$ perfom a shift $x_i\to x_i+ax_{i+1}$.
The result is that $|\mu_{i,i+1}|\le1/2$ for $i=1,\ldots,n-1$.
\item Enter the following {\bf loop}:
Find $i$ such that $b_{i+1}+\mu_{i+1,i}^2b_{i}<\omega b_{i}$.
\begin{itemize}
\item[-] If such $i$ exists, swap $x_{i+1}$ and $x_{i}$ and perform the
shifts $x_{i-1}\to x_{i-1}+ax_i, x_{i}\to x_{i}+a'x_{i+1}$ and $x_{i+1}\to x_{i+1}+a''x_{i+2}$ (if they make sense).
Then repeat the loop.
\item[-] If such $i$ does not exist: we exit the loop
\end{itemize}
\item The form is now partially LLL-reduced. To get an LLL-reduced
form, perform a global shift.
\end{enumerate}

The beauty of the LLL-algorithm is its running time.

\begin{theorem}[LLL]\label{runningtimeLLL} Let $B=\max_{i,j}|q_{ij}|$.
When $\omega<1$ the number of loop-iterations of the algorithm
is bounded above by $n^2\log(n^2B/\mu(Q))/|\log\omega|$.
\end{theorem}

Of course it is a bit strange to have a running time estimate in terms
of the unknown quantity $\mu(Q)$. However, in practice one works
with integer quadratic forms, in which case we have $\mu(Q)\ge1$.

We now give some explicit formula for $q_i$ and $\mu_{ij}$ in terms of the
coefficients $q_{ij}$ of $Q$.

\begin{theorem}\label{determinants}
Let $Q$ be a form in $n$ variables with matrix $(q_{ij})_{i,j=1,\ldots,n}$.
Let $b_i$ and $\mu_{ij}$ be the coefficients corresponding to the descending shape
of $Q$. For each $i,j$ with $1\le i\le j\le n$ we define
$$B_{ij}=\left|
\matrix{
q_{11} & \ldots & q_{1,i-1} & q_{1j}\cr
q_{21} & \ldots & q_{2,i-1} & q_{2j}\cr
\vdots & & \vdots & \vdots\cr
q_{i1} & \ldots & q_{i,i-1} & q_{ij}\cr
}\right|.$$
Then $b_i=B_{i,i}/B_{i-1,i-1}$ for $i=1,\ldots,n$ where we adopt the convention
$B_{00}=1$. Moreover,
$\mu_{ij}=B_{ij}/B_{ii}$ for all $i,j$ with $1\le i<j\le n$.
\end{theorem}

\begin{proof}:
We proceed by induction on $i$. For $i=1$ the statement is straightforward to verify,
all determinants have size $1\times1$. Now let $i>1$ and suppose the statement
holds for $b_{i-1}$ and all $\mu_{i-1,j}$. Let us write
$$Q(x_1,\ldots,x_n)=b_1(x_1+\mu_{12}x_2+\cdots+\mu_{1n}x_n)^2+\tilde{Q}(x_2,\ldots,x_n).$$
Note that the coefficients $\tilde{q}_{ij}$ of $\tilde{Q}$ are given by
$$\tilde{q}_{ij}=q_{ij}-q_{1i}q_{1j}/q_{11}$$
for all $i,j$ with $2\le i\le j\le n$. Denote by $\tilde{B}_{i,j}$ the determinant of
the $(i-1)\times(i-1)$ matrix
$$\pmatrix{
\tilde{q}_{21} & \ldots & \tilde{q}_{2,i-1} & \tilde{q}_{1j}\cr
\tilde{q}_{31} & \ldots & \tilde{q}_{3,i-1} & \tilde{q}_{2j}\cr
\vdots & & \vdots & \cr
\tilde{q}_{i1} & \ldots & \tilde{q}_{i,i-1} & \tilde{q}_{ij}\cr
}.$$
By induction we know that $q_i=\tilde{B}_{i,i}/\tilde{B}_{i-1,i-1}$ and
$\mu_{ij}=\tilde{B}_{ij}/\tilde{B}_{ii}$ for $j>i$. Now consider the definition
of $B_{ij}$ given above. We perform a Gaussian row elimination using the element
$q_{11}$. It is straightforward to see that we get $\tilde{B}_{ij}=B_{ij}/q_{11}$.
This yields the desired formulae for $q_i$ and $\mu_{ij}$.
\end{proof}

\begin{proposition}\label{formulaC}
Let notations be as in the previous theorem. Let
$i<n$ and let $C_i$ be the subdeterminant
of $B_{i+1,i+1}$ obtained by deletion of the $i$-th row and column. Then
$$C_{i}B_{i,i}=B_{i+1,i+1}B_{i-1,i-1}+B_{i,i+1}^2.$$
As a consequence,
$$C_{i}/B_{i-1,i-1}=b_{i+1}+\mu_{i,i+1}^2b_{i}.$$
\end{proposition}

\begin{proof} The identity is an immediate consequence of the following general
fact on determinant. Let $M$ be an $n\times$-matrix. Choose integers $i,j$ such
that $1\le i<j\le n$. By $\tilde{M}$ we denote the $(n-2)\times(n-2)$-matrix obtained
from $M$ by deletion of the $i$-th row and column and the $j$-th row and column.
By $M_{kl}$ we denote the matrix obtained from $M$ by deletion of the $k$-th row
and $l$-th column. Then
$$\det(\tilde{M})\det(M)=\det(M_{ii})\det(M_{jj})-\det(M_{ij})\det(M_{ji}).$$
The proof of this identity is an interesting exercise in determinants.
\end{proof}

\begin{corollary}\label{LLLaltconditions}
With the notations as above the LLL-reducedness conditions
can be written as
\begin{enumerate}
\item $2|B_{ij}|\le B_{ii}$ for all $1\le i<j\le n$.
\item $\omega B_{i,i}\le C_{i}$ for $i=1,\ldots,n-1$.
\end{enumerate}
\end{corollary}

We can now rephrase Proposition \ref{substitutionsI} in terms of the
determinants $B_{ij}$.

\begin{proposition}\label{substitutionsII}
Let $Q$ be a form in $n$ variables and $B_{ij}$ with $1\le i\le j\le n$ its
associated subdeterminants. After application of a substitution we denote
the resulting form by $\tilde{Q}$ and its associated subdeterminants by
$\tilde{B}_{ij}$.

Suppose we apply a shift, that is we replace $x_r$ by
$x_r+ax_s$ for $a\in\bbbz$ and $s>r$. Then the subdeterminants associated
to $\tilde{Q}$ read as follows,
\begin{enumerate}
\item $\tilde{B}_{is}=B_{is}+aB_{ir}$ for $i\le r$
\item $\tilde{B}_{ij}=B_{ij}$ for all other $i,j$.
\item $\tilde{C}_r=C_r+2aB_{rs}+a^2B_{rr}$ if $r=s-1$.
\item $\tilde{C}_i=C_i$ whenever $r\ne s-1$ or $r=s-1$ and $i\ne r$.
\end{enumerate}
Suppose we apply the swap $x_r\leftrightarrow x_{r+1}$.
Then the subdeterminants associated
to $\tilde{Q}$ read as follows,
\begin{enumerate}
\item $\tilde{B}_{rr}=C_{r}$
\item $\tilde{B}_{ir}=B_{i,r+1}$ for all $i<r$.
\item $\tilde{B}_{i,r+1}=B_{ir}$ for all $i<r$.
\item $\tilde{B}_{r,j}=(B_{r,r+1}B_{r,j}+B_{r-1,r-1}B_{r+1,j})/B_{rr}$ for all $j>r+1$
\item $\tilde{B}_{r+1,j}=(B_{r+1,r+1}B_{r,j}-B_{r,r+1}B_{r+1,j})/B_{rr}$ for all $j>r+1$.
\item $\tilde{B}_{ij}=B_{ij}$ for all other $i,j$
\item $\tilde{C}_{r}=B_{rr}$
\item $\tilde{C}_{r-1}=(B_{r-2,r-2}C_r+B_{r-1,r+1}^2)/B_{r-1,r-1}$ if $r>1$.
\item $\tilde{C}_{r+1}=(B_{r+2,r+2}C_r+\tilde{B}_{r+1,r+2}^2)/B_{r+1,r+1}$ if $r<n-1$.
\item $\tilde{C}_i=C_i$ for all $i\ne r-1,r,r+1$.
\end{enumerate}
\end{proposition}

\begin{proof}: These are direct consequences of Proposition \ref{substitutionsI}
and Proposition \ref{formulaC}.
\end{proof}

We can now prove Proposition \ref{runningtimeLLL}. During the LLL-algorithm we
keep track of the product ${\cal B}=B_{11}B_{22}\cdots B_{nn}$. First we derive
a lower bound for ${\cal B}$. Note that $B_{ii}$ is the determinant of the
form $Q(x_1,\ldots,x_i,0,\ldots,0)$ in $i$ variables. Its smallest value is
$\ge\mu(Q)$. So, by Theorem \ref{hermite}, we get $\mu(Q)\le iB_{ii}^{1/i}$
(we used $\gamma_i\le i$). Hence $B_{ii}\ge (\mu(Q)/i)^i$. Take the product
over $i$ to get ${\cal B}\ge (\mu(Q)/n)^{(n(n-1)/2}$. An upper bound for
$B_{ii}$ can be given by $(Bi)^i$ (product of maximal lengths of columns).
Hence ${\cal B}\le (Bn)^{n(n-1)/2}$.

During the LLL-algorithm the value of ${\cal B}$ changes. From Proposition
\ref{substitutionsII} it follows that the $B_{ii}$ do not change after a shift.
After a swap $x_r\leftrightarrow x_{r+1}$ all $B_{ii}$ stay the same, except
$B_{rr}$ which becomes $C_r$. Since the swap is made we apparently have
$C_r<\omega B_{rr}$, henc ${\cal B}$ is multiplied by a factor $<\omega$.
Thus the maximal number of swaps can be bounded by
$${n(n-1)\over 2}{\log(Bn)-\log(\mu(Q)/n)\over |\log\omega|}$$
which yields the desired result.

\section{Geodesic algorithms}\label{algorithm}
Let $\boldalpha=(\alpha_1,\ldots,\alpha_d)\in\bbbr^d$. By Dirichlet's theorem
there exist infinitely many $d+1$-tuples $p_1,\ldots,p_d,q\in\bbbz$ with $q>0$
such that
$$\left|\alpha_i-{p_i\over q}\right|\le {1\over q^{1+1/d}},\quad i=1,\ldots,d.$$
The goal of a continued fraction algorithm is to find such $d+1$-tuples or,
if that is not possible, find approximations that come close to Dirichlet's
inequalities. It is known that classical algorithms, such as the Jacobi-Perron
algorithm, do not attain such quality of approximation. Recall for example
the following theorem in the case $d=2$.
\begin{theorem}[Schweiger]
There exists $\delta>0$ such that for almost all pairs $\alpha_1,\alpha_2$ the
Jacobi-Perron algorithm gives us
$$\left|\alpha_i-{p_i\over q}\right|\le {1\over q^{1+\delta}},\quad i=1,2.$$
\end{theorem}
The optimal value of $\delta$ is not known, but experiments suggest that
$\delta\approx 0.31$. In \cite{lagarias}, Lagarias introduces another idea.
Let $t>0$ and consider the form
$$Q_t(\v x,y)=(x_1-\alpha_1y)^2+\cdots+(x_d-\alpha_dy)^2+ty^2.$$
The important observation by Lagarias is the following.
\begin{proposition}\label{lagariasobservation}
Denote by $|\v x|$ the
Euclidean norm in $\bbbr^d$.
Suppose $\v p\in\bbbz^d$ and $q\in\bbbz_{\ge0}$ minimize
the form $Q_t$. Then we have $q>0$ and
$$|q\boldalpha-\v p|q^{1/d}<\sqrt{d+1}.$$
Consequently, if $q>0$,
$$\left|\alpha_i-{p_i\over q}\right|\le {\sqrt{d+1}\over q^{1+1/d}},\quad i=1,\ldots,d.$$
\end{proposition}
\begin{proof}:
The form $Q_t$ has determinant $t$. So there exist $\v p\in\bbbz^d$ and $q\in\bbbz_{\ge0}$
such that
$$Q_t(\v p,q)=|\v p-\boldalpha q|^2+tq^2\le \gamma_{d+1}t^{1/(d+1)}.$$
Hence $|\v p-\boldalpha q|^2\le \gamma_{d+1}t^{1/(d+1)}$ and $q^{2/d}\le \gamma_{d+1}^{1/d}t^{-1/(d+1)}$.
Their product, and $\gamma_{d+1}<2(d+1)/3$, yield the result.
\end{proof}
\medskip

On this observation one can base the following algorithm to determine simultaneous
rational approximations to $\alpha_1,\ldots,\alpha_d$ with the same denominator.
Without loss of generality we can assume that $|\alpha_i|\le 1/2$ for all $i$.
We initialize with the form
$$Q_t^{(0)}=|\v x-\boldalpha y|^2+ty^2$$
in the variables $x_1,\ldots,x_d,y$ and $t\ge1$. This form is Minkowski reduced
and also LLL-reduced for any $\omega\le1$.
We also define $P^{(0)}$ as the $(d+1)\times(d+1)$ identity matrix. We enter the following loop.
\medskip

{\bf Loop}:\label{geodesicalgorithm}
\begin{itemize}
\item[-] Determine the minimum of the set $\{t|Q_t^{(k)}\ {\rm is\ LLL-reduced}\}$ and call it $t_k$.
\item[-] Perform an LLL-reduction on $Q_{t_k-\epsilon}^{(k)}$ for infinitesimal $\epsilon>0$ and
let $\v x\to A_k\v x$ be the corresponding substitution of variables.
\item[-] Define
$Q_t^{(k+1)}(\v x)=Q_t^{(k)}(A_k\v x)$ and $P^{(k+1)}=P^{(k)}A_k$.
\end{itemize}
\medskip

Remarks:
\begin{enumerate}
\item If we replace the word LLL-reduction with Minkowski reduction in the above algorithm we
get the algorithm of Hermite and Lagarias.
\item For any $k$ we have $Q_t(P^{(k)}\v x)=Q_t^{(k)}(\v x)$. Set $(\v p,q)=P^{(k)}\v e_1$.
Then, as a consequence of Theorem \ref{LLLproperties}(2),
$$|\v p-q\boldalpha|^2+tq^2\le 2^{d/2}t^{1/(d+1)}.$$
This implies that, if $q>0$,
$$|\v p-q\boldalpha|\le 2^{d/4}q^{-1/d}.$$
So the first column of $P^{(k)}$ gives us a simultaneous approximation to $\boldalpha$
with a measure which differs from Dirichlet's approximation by at most a factor
depending only on $d$.
\item We expect that most of the time the substitution-matrix $A_k$ will simply be
a shift or a swap.
\end{enumerate}
In what follows we state a number of properties of the algorithm together with a
number of theorems. Proofs will follow in the nex section.

Verification of Minkowski reducedness involves the verification a finite number of inequalities
which are linear in the coefficients of the form. Although this is certainly not true for
LLL-reducedness, the {\it main observation of this paper} is that the inequalities to be
verified for any $Q_t(M\v x)$ (with $M\in GL(d+1,\bbbz)$)
may not be linear in the coefficients of $Q_t$, but they turn to be {\it linear in} $t$.
Recall the LLL-conditions \ref{LLLaltconditions}.
They are stated in terms of the determinants $B_{ij}$ and
$C_i$ formed out of the coefficients of $Q_t(M\v x)$, where $M\in GL(d+1,\bbbz)$.
Fortunately, these determinants have a form given by the following statement.

\begin{proposition}\label{linearint}
Consider the form $Q_t(\v x)$ defined above and let $M\in GL(d+1,\bbbz)$.
To $Q_t(M\v x)$ we associate the determinants $B_{ij}$ as in Theorem \ref{determinants} and
$C_i$ as in Proposition \ref{formulaC}. Then each of these determinants is in $\bbbz[\alpha_1,
\ldots,\alpha_d,t]$, they are quadratic in $\alpha_1,\ldots,\alpha_d$ and linear in $t$.
Moreover, the coefficient of $t$ is in $\bbbz$.
\end{proposition}

\begin{corollary}Let notations be as above. Then the values of $t>0$ for which
$Q_t(M\v x)$ is LLL-reduced form an interval of the form $[t_0,t_1]\cap\bbbr_{>0}$.
More precisely, the set is either empty, or a point or a closed interval $t_0\le t\le t_1$,
or a half-open interval $0<t\le t_1$.
\end{corollary}

This is a direct consequence of the LLL-conditons \ref{LLLaltconditions} and Proposition
\ref{linearint}. To determine the next value $t_{k+1}$ in the algorithm we simply need to
determine the largest $t<t_k$ such that at least one of the inequalities \ref{LLLaltconditions}
becomes an equality. We then need to perform one or more shifts or swaps (or both).
For each operation we need to update the determinants via the rules given in
Proposition \ref{substitutionsII}. Many of these rules are linear in the determinants but
others are not. For example, consider the rule
$$\tilde{C}_{r-1}=(B_{r-2,r-2}C_r+B_{r-1,r+1}^2)/B_{r-1,r-1}.$$
Write
$$\tilde{C}_{r-1}=u_0t+v_0,\ C_r=u_1t+v_1$$
$$B_{r-2,r-2}=u_2t+v_2,\ B_{r-1,r+1}=u_3t+v_3,
\ B_{rr}=u_4t+v_4$$
with $u_j\in\bbbz$ and $v_j\in\bbbz[\alpha_1,\ldots,\alpha_d]$.
Then it follows from the equations that
$$u_0=(u_1u_2+u_3^2)/u_4,\quad v_0=(u_1v_2+u_2v_1+2u_3v_3-u_0v_4)/u_4.$$
So, although the update rules for the determinants are non-linear, the only
non-linear part consists of division by an integer.

Here is a weak version of Theorem 2.1 in \cite{lagarias}.
\begin{theorem}\label{lagarias21}
If $\boldalpha\not\in \bbbq^d$, the sequence of critical points
$t_1,t_2,\ldots,t_k,\ldots$ is an infinite sequence descending to $0$. If
$\boldalpha\in\bbbq^d$ the sequence $t_1,t_2,\ldots$ terminates at some value
$t_k$.
\end{theorem}

The following statement is actually Theorem 2.2 from \cite{lagarias}, but with a
different proof in the next section.

\begin{theorem}\label{lagarias22}
Suppose that the $\bbbq$-rank of the numbers $1,\alpha_1,\ldots,\alpha_d$ is
$r$. Then for each $i$ the limit $\lim_{t\downarrow0}\mu_i(Q(\v x))$ exists. Moreover,
the limit is zero if $i\le r$ and it is positive if $i>r$.
\end{theorem}

From this theorem and the properties of an LLL-reduced form in Theorem \ref{LLLproperties}
it follows that the algorithm is capable of detecting linear relations between
$1,\alpha_1,\ldots,\alpha_d$. On the other hand there are many properties that
the LLL-based algorithm does not have in common with Lagarias' algorithm. For example,
it does not garantee that it finds Euclidean best approximations.
Also, there is no analogue for Lagrange's theorem for ordinary continued fractions.
Suppose we have an exceedingly good approximation in the sense that $||q\boldalpha||
\le \epsilon q^{-1/d}$ with very small $\epsilon$. Setting $t=\epsilon^2q^{-2(d+1)/d}$
we see that this implies that
$$\mu(Q_t)\le 2\epsilon^{2d/(d+1)}t^{1/(d+1)},$$
much smaller than the expected $\gamma_{d+1}t^{1/(d+1)}$. Suppose that during the algorithm
the value $t$ corresponds with the substitution matrix $P$, i.e. $Q_t(P\v x)$ is LLL-reduced.
Then it follows from Theorem \ref{LLLproperties}(2) that
$$Q_t(P\v e_1)\le 2^{d}\mu(Q_t)\le 2^{d+1}\epsilon^{2d/(d+1)}t^{1/(d+1)} \eqno{{\rm (S).}}$$
Unfortunately we cannot conclude from this that the vector $P\v e_1$ corresponds to
the excellent approximation $(\v p,q)$ we are looking for. However, if
$$Q_t(P\v e_2)> 2^d\max(2\epsilon^{2d/(d+1)}t^{1/(d+1)},Q_t(P\v e_1)),$$
we can conclude that $(\v p,q)=P\v e_1$. This is a consequence of Theorem \ref{LLLproperties}(4)
with $s=2$. So under favourable circumstances we can determine excellent approximations.

It is well-known that the LLL-algorithm has been extremely successful in the explicit
solution of diophantine equations (see \cite{deweger}). The reason is that LLL is capable
of showing the {\it non-existence}
of excellent approximations in the sense that $||q\boldalpha||\le\epsilon q^{-1/d}$ with $q$ less
than a given $Q$. One simply has to verify that inequality (S) above is violated.
However, for this one doesn't need the algorithm sketched above. A direct
application of LLL will do.

If one is only interested in the vector $P^{(k)}\v e_1$, one can skip a number of steps in
the algorithm. From the update formulas one sees that the values of $B_{ii}, B_{i,i+1}$ and
$C_i$ are not affected if we perform a shift $x_r\to x_r+ax_s$ with $s>r+1$. Nor are the
LLL-conditions $2|B_{i,i+1}|\le B_{ii}$ and $\omega B_{i,i}\le C_{i}$ affected. Furthermore,
the substitutionmatrix $A$ corresponding to a shift has the property that $A\v e_1=\v e_1$.
These remarks suggest the following {\it partial continued fraction algorithm}. We initialize
$Q_{t_0}$ and $P^{(0)}$ as before. Then we enter the following loop.
\medskip

{\bf Loop}: Determine the minimum of the set $\{t|Q_t^{(k)}\ {\rm is\ partially\ LLL-reduced}\}$
and call it $t^*_k$.
Perform a partial LLL-reduction on $Q_{t^*_k-\epsilon}^{(k)}$ for infinitesimal $\epsilon>0$ and
let $\v x\to A_k\v x$ be the corresponding substitution of variables. Define
$Q_{t}^{(k+1)}(\v x)=Q_{t}^{(k)}(A_k\v x)$ and $P^{(k+1)}=P^{(k)}A_k$.
\medskip

The resulting sequence $t_1^*,t_2^*,\ldots$ is a subsequence of the sequence $t_1,t_2,\ldots$
we found earlier.
If one wants, one can get an LLL-reduced version for any $t$ by performing an
additional global shift. However, if one is only interested in $P^{(k)}\v e_1$
this is not necessary (shifts do not affect $\v e_1$).

\section{Proofs of statements}
\begin{proof} of Proposition \ref{linearint}. The matrix corresponding to $Q_t$
reads
$$Q_t=\pmatrix{1 & 0 & \ldots & -\alpha_1\cr
0 & 1 & \ldots & -\alpha_2\cr
\vdots & & & \vdots\cr
0 & 0 & \ldots& -\alpha_d\cr
-\alpha_1 & -\alpha_2 & \ldots & t+\alpha_1^2+\cdots+\alpha_d^2\cr}.
$$
We use the same notation $Q_t$ for the matrix. Let us write
$\tau=t+\alpha_1^2+\cdots+\alpha_d^2$. The determinants $B_{ij}$ and
$C_i$ are determinants of matrices which have the following form,
$A^TQ_tB$ where $A,B$ are $(d+1)\times k$ matrices of rank $k$ and
entries in $\bbbz$. There exists invertible $k\times k$-matrices $R,S$
with integer entries such that the last rows of $AR$ and $BS$ have at most
non-zero entry, which is at place $k$. The only entry in the matrix
$R^TA^TQ_tBS$ which contains $\tau$, is the one at place $k,k$. The
entries at places $i,j$ with $1\le i,j\le k-1$ are integers. Hence
the determinant is linear in $\tau$ and the coefficient of $\tau$ is an
integer. This proves the second part of Theorem \ref{linearint}.

To prove the first part it suffices to show it after setting $\tau=0$
(i.e. $t=-\alpha_1^2-\cdots\alpha_d^2$). For this value of $t$ the
matrix $Q_t$ is an integer matrix plus a rank 2 matrix with entries that
are linear in $\alpha_1,\ldots,\alpha_d$. The same holds for the matrix
$A^TQ_tB$. Hence its determinant is a quadratic polynomial in the $\alpha_i$
with integer coefficients.
\end{proof}

For the proof of Theorem \ref{lagarias21} we need a Lemma.
\begin{lemma}Let $t_0>0$. Then the number of $M\in GL(d+1,\bbbz)$ such that
$Q_t(M\v x)$ is LLL-reduced for some $t\ge t_0$, is finite.
\end{lemma}

\begin{proof}: Note that the successive minima $\mu_i(Q_t)$ are decreasing in $t$.
Let $t\in[t_0,1]$ and suppose $M\in GL(d+1,\bbbz)$ is such that $Q_t(M\v x)$ is
LLL-reduced. Consider the following estimates for $i=1,2,\ldots,d+1$,
$$Q_{t_0}(M\v e_i)\le Q_t(M\v e_i)\le 2^{d}\mu_i(Q_t)\le 2^{d}\mu_i(Q_1).$$
The middle estimate follows from Theorem \ref{LLLproperties}(4). Since the inequality
$Q_t(\v x)\le 2^d\mu_i(Q_1)$ in $\v x\in\bbbz^{d+1}$ has finitely many solutions, there are
finitely many possibilities for $M\v e_i$, the $i$-th column of $M$. Hence our
lemma follows.
\end{proof}
\medskip

\begin{proof} of Theorem \ref{lagarias21}. Recall that the $t_i$ form a decreasing sequence.
Suppose there is a point of accumulation $t_{\infty}>0$. Hence we have infinitely many
$t_k>t_{\infty}$ such that $Q_{t_k}(P^{(k)}\v x)$ is LLL-reduced. Since the $P^{(k)}$ are
distinct, this contradicts the lemma we just proved.So we have either $\lim_{k\to\infty}t_k=0$
or the sequence $t_1,t_2,\ldots$ stops at $t_k$. In the latter case the form $Q_t(P^{(k)}\v x)$
is LLL-reduced for all $t$ with $0<t<t_k$. Since $Q_t(P^{(k)}\v e_1)\le Ct^{1/(d+1)}$ for
some $C>0$ we see that $Q_0(P^{(k)}\v e_1)=0$. Letting $(p_1,\ldots,p_d,q)=P^{(k)}\v e_1$ this
implies that $(p_1-\alpha_1 q)^2+\cdots+(p_d-\alpha_d q)^2=0$. Hence $p_i/q=\alpha_i$
for $i=1,\ldots,d$. So all $\alpha_i$ are rational.

Similarly, if the sequence of $t_i$ is infinite, we let $(\v p^{(k)},q)=P^{(k)}\v e_1$
and see that $|\v p^{(k)}-\boldalpha q|\to0$ as $k\to\infty$. This is only possible
if not all $\alpha_i$ are rational.
\end{proof}
\medskip

\begin{proof} of Theorem \ref{lagarias22}. It suffices to prove that
$\lim_{t\downarrow0}\mu_{r+1}(Q_t)>0$ if $r\le d$ and $\lim_{t\downarrow0}\mu_r(Q_t)=0$.

Let $L\subset \bbbz^{d+1}$ be the lattice of vectors
$(l_0,\v l)$ such that $l_0+\v l\cdot\boldalpha=0$.
It has $\bbbz$-rank $d+1-r$. We choose a fixed basis $B$.
Suppose we have $r+1$ independent vectors
$(\v p_i,q_i)\in \bbbz^{d+1}$ such that
$Q_t(\v p_i,q_i)\le \mu_{r+1}(Q_t)$ for $i=1,\ldots,r+1$.
Then there exists a vector $(l_0,\v l)\in B$ and an $i$ such that
$l_0q_i+\v l\cdots\v p_i\ne0$. Hence
$$1\le|l_0q_i+\v l\v p_i|=|\v l\cdot(\v p_i-q_i\boldalpha)|.$$
This implies $|\v p_i-q_i\boldalpha|\ge 1/|\v l|$ and hence
$$\mu_{r+1}(Q_t)\ge Q_t(\v p_i,q_i)\ge 1/|\v l|^2.$$
This proves the first part of Theorem \ref{lagarias22}.

We may assume without loss of generality that $1,\alpha_1,\ldots,\alpha_{r-1}$
are $\bbbq$-linear independent. For any $j\ge 1$ we write
$\alpha_j=a_{j0}+\sum_{k=1}^{r-1}a_{jk}\alpha_k$ and let $N$ be the common
denominator of the $a_{jk}$.

{\bf Claim}: to any $\epsilon>0$ there
exists a rank $r$ matrix $P=(p_{ij})_{i=1,\ldots,r;j=0,\ldots,r-1}$ such
that $|p_{ij}-\alpha_jp_{i0}|\le \epsilon$ for every $i=1,\ldots,r;j=1,\ldots,r-1$.

We then proceed as follows. Define $p_{ij}=a_{j0}p_{i0}+\sum_{k=1}^{r-1}
a_{jk}p_{ik}$ for any $j\ge 1$ and $i=1,\ldots,r$. Then
\begin{eqnarray*}
|Np_{ij}-N\alpha_jp_{i0}|&=&\left|\sum_{k=1}^{r-1}Na_{jk}(p_{ik}-\alpha_kp_{i0})\right|\\
&\le& rNA\epsilon
\end{eqnarray*}
where $A=\max_{i,j}|a_{jk}|$. Write $p_{i0}=q_i$ and $\v p_i=(p_{i1},\ldots,p_{id})$
for $i=1,\ldots,r$. Note that $N\v p_i\in\bbbz^d$ and $Nq_i\in\bbbz$.
Then,
$$\mu_r(Q_t)\le\max_i(|N\v p_i-q_i\boldalpha|^2+tq_i^2)\le d(rNA\epsilon)^2+t\max_iq_i^2.$$
Letting $t$ go to $0$ this implies $\lim_{t\downarrow0}\mu_r(Q_t)\le d(rNA\epsilon)^2$.
Since $\epsilon$ can be chosen arbitrarily small our second assertion follows.
\end{proof}

\end{document}